\documentclass[a4paper,12pt]{article}

\usepackage[utf8]{inputenc}
\usepackage{comment}
\usepackage{float}
\usepackage{amstext}
\usepackage{amsfonts}
\usepackage{amsthm}
\usepackage{textcomp}
\usepackage{amssymb}
\usepackage{amscd}
\usepackage{amsmath}
\usepackage{xcolor}
\usepackage{tikz-cd}
\tikzcdset{scale cd/.style={every label/.append style={scale=#1},
    cells={nodes={scale=#1}}}}
\usepackage{caption}
\usepackage{enumitem}
\usepackage{theoremref}
\setlist[enumerate]{label=({\arabic*})}

\newtheorem{defn}{Definition}[subsection]
\newtheorem{thm}[defn]{Theorem}

\newtheorem{rmk}[defn]{Remark}
\newtheorem{prop}[defn]{Proposition}
\newtheorem{cor}[defn]{Corollary}

\newtheorem{manualtheoreminner}{Theorem}
\newenvironment{manualtheorem}[1]{%
  \renewcommand\themanualtheoreminner{#1}%
  \manualtheoreminner
}{\endmanualtheoreminner}

\newtheorem{manualcorinner}{Corollary}
\newenvironment{manualcor}[1]{%
  \renewcommand\themanualcorinner{#1}%
  \manualcorinner
}{\endmanualcorinner}

\newcommand{\N}{\mathbb{N}}

\newcommand{\Z}{\mathbb{Z}}

\newcommand{\B}{{\mathcal B}}

\newcommand{\D}{{\mathbb D}}

\DeclareMathOperator*\lowlim{\underline{lim}}
\DeclareMathOperator*\uplim{\overline{lim}}

\begin{document}
\title{Expansivity and strong structural stability for composition operators on $L^p$ spaces} 

\author{Martina Maiuriello}

\newcommand{\Addresses}{{
  \bigskip
  \footnotesize

  M.~Maiuriello,\\
  \textsc{Dipartimento di Matematica e Fisica,\\ Universit\`a degli Studi della Campania ``Luigi Vanvitelli",\\
  Viale Lincoln n. 5, 81100 Caserta, Italy}\\
  \textit{E-mail address: \em martina.maiuriello@unicampania.it}

}}

\date{}
\maketitle

\begin{abstract}
In this note we investigate the two notions of expansivity and strong structural stability for composition operators on $L^p$ spaces, $1\leq p < \infty$. Necessary and sufficient conditions for such operators to be expansive are provided, both in the general and the dissipative case. We also show that, in the dissipative setting, the shadowing property implies the strong structural stability and we prove that these two notions are equivalent under the extra hypothesis of positive expansivity.  
\end{abstract}
\let\thefootnote\relax\footnote{\vspace{-0.3cm}\\ {\em 2010 MSC:} Primary: 47B33, 37C20; Secondary: 37B05, 37B65.\\
{\em Keywords:} $L^p$ spaces, composition operators, expansivity, shadowing property, structural stability.\\
\vspace{-0.2cm} \\
This research has been partially supported by the INdAM group GNAMPA ``Gruppo Nazionale per l’Analisi Matematica, la Probabilità e le loro Applicazioni'' and the project Vain-Hopes within the program VALERE: VAnviteLli pEr la RicErca; and partially accomplished within the UMI group TAA ``Approximation Theory and Applications''. }

\section{Introduction}

\indent The notion of structural stability, which comes from the work of Andronov and Pontrjagin \cite{AP}, is one of the fundamentals in the theory of linear dynamical systems. In recent years, many questions in hyperbolic dynamics, concerning relations between structural stability, expansivity, shadowing property and hyperbolicity, have been deeply analyzed and, sometimes, completely answered. Taking a glance at some key results in this field, one may observe that it is proved by Abdenur and Díaz, in \cite{AbdenurDiaz2007}, that, for $C^1$ diffeomorphisms on closed manifolds,  in certain contexts, the shadowing property implies hyperbolicity, and therefore structural stability. Pilyugin and Tikhomirov \cite{PilyuginTikhomirov2010} showed that, for $C^1$ diffeomorphisms of closed smooth manifolds, structural stability is equivalent to Lipschitz shadowing property. In \cite{BernardesMessaoudi}, Bernardes and Messaoudi  proved that all generalized hyperbolic operators are structurally stable.
The same authors, in \cite{BM}, showed, among other results, that hyperbolicity is equivalent to expansivity plus shadowing and that hyperbolicity implies expansivity, shadowing property and strong structural stability while, in general, each of the converses is false. \\
As it often happens in linear dynamics, many researchers have widely analyzed the above mentioned notions in the context of a special class of operators, the {\em weighted shifts}, and this is due to their versatility in the construction of examples in linear dynamics, in operator theory and its applications. Therefore, in the last decades, many dynamical properties have been completely analyzed and characterized for such operators \cite{Bayart2021, BCDMP,  BM, DAnielloDarjiMaiuriello}, sometimes even before the property in question was completely understood  in more general contexts. In particular, in \cite{BCDMP} and \cite{Bayart2021}, the authors contribute to this line of research by providing characterizations of expansive and strongly structurally stable weighted shifts, respectively. It turns out from the literature that weighted shifts represent a good model for understanding the dynamics of a natural class of operators: the {\em composition operators} on $L^p$ spaces, $1 \leq p < \infty$. 
Although some dynamical properties have been fully understood for composition operators \cite{BADP, BDP}, unfortunately, many other notions require the extra hypotheses of dissipativity and bounded distortion in order to be characterized for such operators: this is the case, among others, of generalized hyperbolicity and shadowing property  \cite{DAnielloDarjiMaiuriello, Darjipires}. More is true: recently, in \cite{DAnielloDarjiMaiuriello2}, focusing on chaotic properties as well as hyperbolic properties (such as shadowing, expansivity and generalized hyperbolicity), the authors established that dissipative composition operators with bounded distortion are {\em shifts-like operators}, in the sense that they behave similarly to weighted shifts. On the other hand, up to now, no characterization of structural stability and strong structural stability is known for this large class of operators. \\

\vspace{-0.2cm}
Motivated by these results, in this note, we provide a characterization of various types of expansivity for composition operators, both in the general and dissipative context. Moreover, we start the investigation of strong structural stability for these operators. \\
The note is organized as follows. In Section 2, the notation is fixed and preliminary definitions and background results are recalled. In Section 3, a characterization of expansive composition operators is stated and proved. Section 4 is  devoted to strongly structurally stable composition operators. It is there proved that, in the dissipative setting, the shadowing property implies the strong structural stability and that these two notions are equivalent under the extra hypothesis of positive expansivity. Anyway, the theory concerning structural stability and strong structural stability, for these class of operators, is far from being complete.

\section{Preliminary Definitions and Results}
\indent Throughout the paper, ${\mathbb N}$ denotes the set of all positive integers and ${\mathbb N}_0={\mathbb N}\cup \{0\};$  ${\mathbb D}$ and ${\mathbb T}$ are, respectively, the open unit disk and the unit circle in the complex plane ${\mathbb C}$. The space $X$ is always assumed to be a complex Banach space, in which the {\em unit sphere} is denoted by $S_X=\{x \in X\, : \, \Vert x \Vert =1 \}$. In the sequel, unless otherwise stated, $T$ denotes an invertible bounded linear operator from $X$ to itself and, as usual,  $\sigma(T)$ represents its {\em spectrum}. Moreover,  $C_{b}(X)$ denotes the Banach space of all bounded continuous maps $\phi : X  \rightarrow X$ endowed with the supremum norm ${\Vert \cdot \Vert}_{\infty}$.  \\

\subsubsection*{Expansivity, Shadowing and Hyperbolicity}

\begin{defn}
 The operator $T$ 
\begin{itemize}
    \item is  {\em (positively) expansive} if for each $x$ with $\Vert x \Vert =1$, there exists $n \in \mathbb Z$ $(n \in \mathbb N)$ such that $\Vert T^n x\Vert \geq 2$;
    \item  is  {\em (positively) uniformly expansive} if there exists $n \in \mathbb N$ such that, for every $x \in X$ with $\Vert x \Vert =1$,  $\Vert T^nx \Vert \geq 2$  or  $\Vert T^{-n}x \Vert \geq 2$ {\em{(}}for every $x \in X$ with $\Vert x \Vert =1$,  $\Vert T^nx \Vert \geq 2${\em{)}};
    \item has the {\em shadowing property} if, for each $\epsilon >0$, there exists $\delta >0$ such that every sequence  $\{x_n\}_{n \in 
 {\mathbb Z}}$  with \[ \Vert Tx_n - x_{n+1} \Vert \leq \delta, \text{ for all $n \in \mathbb Z$,}\] called $\delta -$pseudotrajectory of $T$,  is $\epsilon -$shadowed by a real trajectory of $T$, that is, there exists $x \in X$ such that \[ \Vert T^n x-x_n \Vert < \epsilon, \text{ for all $n \in  {\mathbb Z}$;}\]
 \item is {\em hyperbolic} if $\sigma (T) \cap {\mathbb T}= \emptyset$;
 \item is {\em generalized hyperbolic} if $X=M\oplus N$, where $M$ and $N$ are closed subspaces of $X$ with $T(M)\subseteq M$ and $T^{-1}(N) \subseteq N$,  and $\sigma(T_{|_M}) \subset \D$ and $\sigma(T^{-1}_{|_N}) \subset \D$.
   \end{itemize}
\end{defn}

We recall that in the above definition, in the cases of positive expansivity and positive uniform expansivity, the operator $T$ does not need to be invertible. Relations between these notions have been widely investigated in the literature (see, for instance, \cite{BCDMP, BernardesMessaoudi, BM, CGP, EH, E, O}, and references therein). In particular, it is known that hyperbolicity is equivalent to expansivity plus shadowing property and that hyperbolicity implies shadowing property and uniform expansivity, but the converse is false in general \cite{BM}.

\subsubsection*{Structural Stability and Strong Structural Stability}

There are several versions of structural stability in the literature: here it is considered the following one, originally given by Pugh in \cite{PU}. 

\begin{defn}
An invertible operator $T$ on $X$ is said to be {\em structurally stable} if there exists $\epsilon >0$ such that $T + \phi$ is topologically conjugate to $T$ whenever $\phi \in C_{b}(X)$ is a Lipschitz map with ${\Vert \phi \Vert}_{\infty}\leq \epsilon$ and $Lip(\phi) \leq \epsilon$. 
\end{defn}

The following stronger notion, introduced in  \cite{R}, is obtained by requiring extra properties on the conjugation between $T$ and $T + \phi$.

\begin{defn}
An invertible operator $T$ on  $X$ is said to be {\em strongly structurally stable}  if for every $\gamma >0$ there exists $\epsilon >0$ such that the following property holds: for any Lipschitz map $\phi \in C_{b}(X)$ with  ${\Vert \phi \Vert}_{\infty}\leq \epsilon$ and $Lip(\phi) \leq \epsilon$, there is a homeomorphism $h : X \rightarrow X$  such that $h \circ T = (T + \phi) \circ h$ and ${\Vert h - I \Vert}_{\infty} \leq \gamma$, namely the homeomorphism $h$ conjugating $T$ and $T + \phi$ is close to the identity operator. 
\end{defn}

The following result holds.

\begin{thm} \cite[Theorem 1]{BernardesMessaoudi} \label{genhyp}
Every generalized hyperbolic operator on a Banach space is strongly structurally stable.
\end{thm}

\noindent
Relations between expansivity and structural stability are investigated in \cite{BM}.

\begin{thm} \cite[Theorem 6]{BM} \label{SS}
Let $T$ be an invertible operator on $X$. Assume that $T$ is structurally stable. Then, the following hold:
\begin{enumerate}
\item[(a)] {If T is expansive, then T is uniformly expansive.}
\item[(b)] {If T is positively expansive, then T is hyperbolic.}
\end{enumerate}
\end{thm}

The basic relations between all the above mentioned notions are summarized in the following diagram, in which, in general, none of the implications can be reverted, as showed in  \cite{BCDMP, EH} and \cite[Theorem 9]{BM}.

\begin{figure}[H]
\centering
\begin{tikzcd}[sep=small,arrows=Rightarrow, scale cd=.9]
 &  & \text{Uniform expansivity} \arrow[r] &  \text{Expansivity} \\
& \text{Hyperbolicity}\arrow[ur] \arrow[rd] &  \\
&  &\text{Generalized hyperbolicity} \arrow[d] \arrow[r] &  \text{Shadowing property}\\
& & \text{Strong structural stability} \arrow[d] \\
&  & \text{Structural stability}
\end{tikzcd}
\end{figure}

\subsection{Background Results for Weighted Shifts}

Weighted shifts were introduced in \cite{Rolewicz}.

\begin{defn} 
Let $X=\ell^p(\Z)$, $1 \leq p < \infty$ or $X=c_0(\Z).$ Let  $w=\{w_n\}_{n \in \Z}$ be a bounded sequence of scalars, called {\em weight sequence}.  Then, the {\em bilateral weighted backward shift $B_w$ on $X$} is defined by \[B_w(\{x_n\}_{n \in \Z}) =\{w_{n+1}x_{n+1}\}_{n \in  \Z}.\] 
\end{defn}

\noindent The boundedness of $w=\{w_n\}_{n \in \Z}$ is a necessary and sufficient condition for $B_w$ to be a well-defined bounded operator on $X$. A bilateral $B_w$ is invertible if and only if  $\inf_{n \in \mathbb Z} \vert w_n \vert >0$. Many dynamical properties mentioned in Section 2  have been completely analyzed for such operators: characterizations of expansivity and shadowing are proved, respectively, in \cite[Theorem E]{BCDMP} and \cite[Theorem 18]{BM}; hyperbolic and generalized hyperbolic weighted shifts are characterized in \cite[Theorem 2.4.5]{DAnielloDarjiMaiuriello}. Recently, strong structural stability has also been investigated, as the following result shows.

\begin{thm} \cite[Theorem 1.2]{Bayart2021} \label{theoSSBW}
Let  $X=\ell^p({\mathbb Z})$ $(1 \leq p < \infty)$ or $X=c_0({\mathbb Z})$. Let $B_w$ be an invertible bilateral weighted backward shift. Then $B_w$ is strongly structurally stable if and only if one of the following conditions holds:
\begin{itemize}
\item[a)]{$\displaystyle{\lim_{n \rightarrow \infty} \left ( \sup _{k \in {\mathbb Z}} (\prod_{j=k}^{k+n}\vert w_{j}\vert)^{\frac{1}{n}} \right )<1;}$}
\item[b)]{$\displaystyle{\lim_{n \rightarrow \infty} \left ( \inf_{k \in {\mathbb Z}} (\prod_{j=k}^{k+n}\vert w_{j}\vert)^{\frac{1}{n}} \right )>1;}$}
\item[c)]{$\displaystyle{\lim_{n \rightarrow \infty} \left (\sup _{k \in {\mathbb N}} (\prod_{j=-k-n}^{-k}\vert w_{j}\vert)^{\frac{1}{n}}\right )<1}$ \ \ and \ \ $\displaystyle{\lim_{n \rightarrow \infty}\left (\inf_{k \in {\mathbb N}} (\prod_{j=k}^{k+n}\vert w_{j}\vert)^{\frac{1}{n}} \right )>1.}$  }
\end{itemize}
\end{thm}

\begin{cor}\cite[Corollary 1.3]{Bayart2021} \label{theoSSBW}
Let  $X=\ell^p({\mathbb Z})$ $(1 \leq p < \infty)$ or $X=c_0({\mathbb Z})$. Let $B_w$ be an invertible bilateral weighted backward shift.  Then $B_w$ is strongly structurally stable if and only if it has the shadowing property.
\end{cor}

Using the previous result together with \cite[Theorem 1]{BM}, it follows that, for weighted shifts, hyperbolicity is equivalent to expansivity plus strong structural stability.

\subsection{Composition Operators}

\indent The setting, in which all the results of this note are proved, is fixed in the following definition.

\begin{defn}\label{compodyn}
A {\em composition dynamical system} is a quintuple $(X,{\mathcal B},\mu, f, T_f)$ where
\begin{enumerate}
     \item $(X,{\mathcal B},\mu)$ is a $\sigma$-finite measure space, 
    \item $f : X \to X$ is an injective {\em bimeasurable transformation},
i.e., $f(B) \in {\mathcal B}$ and $f^{-1}(B) \in {\mathcal B}$ for every $B \in {\mathcal B}$,
\item there is  $c > 0$ such that
\begin{equation}\label{condition}
   \mu(f^{-1}(B)) \leq c \mu(B) \ \textrm{ for every } B \in {\mathcal B},
   \tag{$\star$}
\end{equation}
\item $T_f: L^p(X) \rightarrow L^p(X) $, $1 \le p <\infty$, is the {\em composition operator} induced by $f$, i.e.,
\[T_f : \varphi \mapsto \varphi \circ f.\] 
\end{enumerate}
\end{defn}

\noindent It is well-known that (\ref{condition}) guarantees that $T_f$ is a bounded linear operator. Moreover, if $f$ is surjective and (\ref{condition}) holds with $f^{-1}$ replaced by $f$, then $T_{f^{-1}}$ is a well-defined bounded linear operator and $T^{-1}_f = T_{f^{-1}}$. For a detailed exposition on composition operators, see \cite{SM}. As it turns out from the literature, studying a dynamical property in the general context of composition operators is sometimes complicated. For instance, properties like shadowing, generalized hyperbolicity, chaos and frequent hypercyclicity are characterized, up to now, for composition operators with additional conditions: the dissipativity and the bounded distortion. 

\subsection{Dissipativity and Bounded Distortion}

\indent From now on, the measurable space $(X,{\mathcal B},\mu)$ is always assumed to be $\sigma$-finite. We recall that the transformation $f: X \rightarrow X$ is said to be {\em{non-singular}} if ``for each $B \in \mathcal B$, $\mu(f^{-1}(B))=0$ if and only if $\mu(B)=0$''.  Below, only the relevant definitions needed in the sequel are given. The reader interested in the topic may refer to \cite{AaronsonMSM1997,DAnielloDarjiMaiuriello,K} for more details and an exhaustive exposition on how they naturally arise from the Hopf Decomposition Theorem.

\begin{defn} Let $(X, {\mathcal B}, \mu)$ be a measure space and $f: X \rightarrow X$ be an invertible non-singular transformation. A measurable set $W \subset X$ is called a {\em  wandering set (for $f$) } if the sets $\{f^{-n}(W)\}_{n \in \Z}$ are pairwise disjoint.  
\end{defn}

\begin{defn} \label{dissipcompodyn}
Let $(X,{\mathcal B},\mu)$ be a measure space and $f:X \rightarrow X$ be invertible and non-singular. The quadruple $(X,{\mathcal B},\mu, f)$ is called 
\begin{itemize}
\item a {\em dissipative system generated by $W$,} if  $X = \dot {\cup} _{k \in \Z} f^k (W)$ for some $W \in \B$ with $0 < \mu (W) < \infty$ $($the symbol $\dot {\cup} $ denotes pairwise disjoint union$)$;
\item a {\em dissipative system, of bounded distortion, generated by $W$,} if there exists $K>0$ such that
\begin{equation*}
 \dfrac{1}{K} \mu(f^k(W))\mu(B) \leq \mu(f^k (B))\mu (W) \leq K \mu(f^k(W))\mu(B), \tag{$\Diamond$}\label{eq:conditionbd}
\end{equation*}
for all $k \in \mathbb Z$ and  $B \in {\mathcal B}(W)$, where  ${\mathcal B}(W) =\{ B \cap W, B \in {\mathcal B} \}.$
\end{itemize}
\end{defn}

\begin{defn}
A composition dynamical system $(X,{\mathcal B},\mu, f, T_f)$ is called
\begin{itemize}[series=condition]
\item  {\em dissipative composition dynamical system, generated by $W$,} if $(X,{\mathcal B},\mu, f)$ is a dissipative system generated by $W$;
\item {\em dissipative composition dynamical system, of bounded distortion, generated by $W$,} if $(X,{\mathcal B},\mu, f)$ is a dissipative system of bounded distortion, generated by $W$.
\end{itemize}
\end{defn}

In the sequel, the following result is needed.

\begin{prop} \cite[Proposition 2.6.5]{DAnielloDarjiMaiuriello}\label{diststar}
Let $(X,{\mathcal B},\mu, f)$ be a dissipative system of bounded distortion generated by $W$. Then, there is a constant $H>0$ such that, for all $B \in {\mathcal B}(W)$ with $\mu(B)> 0$ and for each $s, t \in \Z$,
\begin{equation*} 
  \dfrac{1}{H} \dfrac{\mu(f^{t+s}(W))}{\mu(f^s(W))} \leq \dfrac{\mu(f^{t+s} (B))}{\mu (f^s(B))} \leq H \dfrac{\mu(f^{t+s}(W))}{\mu(f^s(W))}. \tag{$\Diamond \Diamond$} \label{eq:generalbd}
\end{equation*} 
\end{prop}

\section{Expansivity for Composition Operators}

From now on, ${\mathcal B}^+=\{ B \in {\mathcal B} : 0<\mu(B)<\infty\}.$ 

\begin{manualtheorem}{E} \label{PROPEX1}
Let $(X,{\mathcal B},\mu, f, T_f)$ be a composition dynamical system. The following statements hold.
\begin{enumerate}
\item $T_f$ is positively expansive if and only if for each $B \in \mathcal B$ with positive measure, \[ \sup_{n \in \Bbb N}\mu (f^{-n}(B))=\infty .\]
\item $T_f$ is expansive if and only if for each $B \in \mathcal B$ with positive measure,  \[ \sup_{n \in \Bbb Z}\mu (f^{-n}(B))= \infty .\]
\item $T_f$ is uniformly positively expansive if and only if  \[\lim_{n \rightarrow  \infty} \dfrac{ \mu (f^{-n}(B))}{\mu(B)}= \infty\] uniformly with respect to $B \in {\mathcal B}^+$. 
\item{$T_f$ is uniformly expansive if and only if ${\mathcal B}^+$ can be splitted as ${\mathcal B}^+={\mathcal B}^+_{\mathcal A} \cup {\mathcal B}^+_{\mathcal C}$ where 
\[\lim_{n \rightarrow  \infty} \dfrac{\mu (f^{n}(B))}{\mu(B)}= \infty, \ \ \text{ uniformly on } {\mathcal B}^+_{\mathcal A},\]  
\[\lim_{n \rightarrow  \infty} \dfrac{\mu (f^{-n}(B))}{\mu(B)}= \infty, \ \ \ \text{ uniformly on } {\mathcal B}^+_{\mathcal C}.\] }
\end{enumerate}
\end{manualtheorem}

\begin{proof}
The proof of (1) is obtained replacing $\Z$ by $\N$ in the proof of (2). \\

(2). Assume $T_f$ expansive, i.e., using statement $c)$ of \cite[Proposition 19]{BCDMP}, 
\[\sup_{n \in \Bbb Z} \Vert T_f^n \varphi \Vert_p = \infty, \ \text{ for each } \ \varphi \in L^p(X)\setminus \{0\}.\] Let $B \in {\mathcal B}$ with $\mu (B)>0,$ and take $\varphi=\chi _B.$ Note that, for every $n \in {\Bbb Z}$,
\[ \Vert T_f^n \varphi \Vert_p ^p=\Vert \varphi \circ f^n \Vert_p ^p = \int_X \vert \varphi \circ f^n \vert ^p d\mu= \int_X \vert \chi_B \vert ^p \circ f^n d\mu = \mu (f^{-n}(B)),\]
implying  $\displaystyle{\sup_{n \in \Bbb Z} }\  \mu (f^{-n}(B)) =  \infty$, and hence the thesis. \\
Conversely, assume $\displaystyle{\sup_{n \in \Bbb Z} } \ \mu (f^{-n}(B))=\infty$ for each $B \in {\mathcal B}$ with $\mu (B)>0 $.
Let $\varphi \in  L^p(X) \setminus \{0\}$. Then, there exists $\delta >0$ such that the set $B'=\{ x \in X \,:\, \vert \varphi (x) \vert >\delta \}$ has positive measure. For each $n \in \Z$, \[\Vert T_f^n \varphi \Vert_p^p=\int_X \vert \varphi \circ f^n \vert ^p d\mu \geq \int_{f^{-n}(B')} \vert \varphi \circ f^n \vert ^p d\mu \geq \delta ^p \mu (f^{-n}(B')),\]
implying $\displaystyle{\sup_{n \in \Bbb Z} }\Vert T_f^n \varphi \Vert_p = \infty .$ By the arbitrariness of $\varphi \in L^p(X)\setminus \{0\}$ and applying statement $c)$ of \cite[Proposition 19]{BCDMP}, it follows that $T_f$ is expansive. \\

(3). Assume $T_f$ uniformly positively expansive, i.e., by $b)$ of \cite[Proposition 19]{BCDMP}, 
\[ \lim_{n \rightarrow \infty} \Vert T_f^n \varphi \Vert_p = \infty, \text{ uniformly on } S_{L^p(X)}\]
where we recall that  $S_{L^p(X)}=\{ \varphi \in L^p(X) : \Vert \varphi \Vert _p =1\}.$
For each $B \in {\mathcal B}^+,$ take $\varphi=\frac{\chi _B}{\mu (B)^{\frac{1}{p}}}$ and note that, for each $n \in \N$,
\[ \Vert T_f^n \varphi \Vert_p ^p= \int_X \vert \varphi \circ f^n \vert ^p d\mu= \int_X \frac{{\vert \chi_B \vert}^p}{\mu (B)}  \circ f^n d\mu = \dfrac{\mu (f^{-n}(B))}{\mu (B)}.\] 
As $\varphi \in S_{L^p(X)}$, this implies  
\[\lim_{n \rightarrow  \infty} \dfrac{ \mu (f^{-n}(B))}{\mu(B)}= \infty, \text{ uniformly on the sets } B \in {\mathcal B}^+.\]
To prove the converse, according to statement $b)$ of \cite[Proposition 19]{BCDMP}, it will be shown that $\displaystyle{\lim_{n \rightarrow \infty} }\Vert T_f^n \varphi \Vert_p = \infty$ uniformly on $S_{L^p(X)}.$ It is enough to prove it for simple functions in $S_{L^p(X)}$ and then to use approximation. Let $M>0$. By hypothesis, in correspondence of $M$ there exists ${\overline n} \in \Bbb N $ such that for each $B \in {\mathcal B}^+,$
\[ \dfrac{\mu (f^{-n}(B))}{\mu(B)} > M, \ \ \  \text{ for each } n \geq \overline n. \vspace{-0.02cm} \] 
Let $\varphi \in S_{L^p(X)}$ be a simple function, that is $\varphi = \sum_{i=1}^{s} {\alpha}_{i} {\chi}_{B_{i}}$ where ${\alpha}_{i} \in {\Bbb C}\setminus \{0\}$, the $B_{i}$'s are pairwise disjoint measurable sets, and $\Vert \varphi \Vert _p^p = \sum_{i=1}^{s} \vert{\alpha _i}\vert^p {\mu (B_{i})} =1.$ Without loss of generality, assume $B_i \in {\mathcal B}^+$ for each $i \in \{1,...,s\}.$ For each $n \geq \overline n$, 
\begin{eqnarray*}
 \Vert T_f^n \varphi \Vert_p^p =  \int_X \vert \varphi \circ f^n \vert ^p d\mu &  = & \sum_{i=1}^{s} \int_X \vert \alpha _i \vert ^p \vert \chi_{B_i} \vert^p \circ f^n d\mu \\
 &=&  \sum_{i=1}^{s} \vert{{\alpha}_{i}}\vert ^p \mu (f^{-n}(B_{i}))\\
&>& \sum_{i=1}^{s} \vert{{\alpha}_{i}}\vert ^p M \mu (B_i)   \\
& = &  M \sum_{i=1}^{s} \vert{{\alpha}_{i}}\vert ^p \mu (B_{i})\\
& = &   M \Vert \varphi \Vert _p ^p \\ 
& = &  M.
\end{eqnarray*}
This shows that, for every $M>0$, there exists ${\overline n } \in \Bbb N$ such that, for each simple function $\varphi \in S_{L^p(X)},$
\[\Vert T^n_f \varphi \Vert _p^p >M, \ \ \  \text{ for each } n \geq \overline n,\] 
i.e.,
\[\lim_{n \rightarrow \infty} \Vert T_f^n \varphi \Vert_p = \infty.\]
Now, let $\varphi$ be an arbitrary element of $S_{L^p(X)}$. Write $\varphi = {\varphi}^{+} - {\varphi}^{-}$, where $ {\varphi}^{+}$ and $ {\varphi}^{-}$ are the positive and the negative part of $\varphi$, respectively. 
Let $\{{\varphi}^{+}_{k}\}_{k \in \N}$ and $\{{\varphi}^{-}_{k}\}_{k \in \N}$ be two non-decreasing sequences of simple functions, pointwise converging to ${\varphi}^{+}$ and ${\varphi}^{-}$, respectively. Then, the sequence  $\{{\varphi}_{k}\}_{k \in \N}$ defined as ${\varphi}_{k} = {\varphi}^{+}_{k} - {\varphi}^{-}_{k}$, pointwise converges to $\varphi$, and 
$\vert {\varphi}_{k} \vert  = \vert {\varphi}^{+}_{k} - {\varphi}^{-}_{k} \vert \leq 2 \vert \varphi \vert.$ 
By the Lebesgue Dominated Convergence Theorem, 
\[\lim_{k \rightarrow \infty} \Vert \varphi _k \Vert _p = \Vert \varphi  \Vert _p =1,\]
and hence, there exists $k_{0} \in \N$ such that, for each $k > k_{0}$,  
\[\Vert \varphi _k \Vert _p > \frac{1}{2}, \]
implying, for each $k > k_{0}$, $n \in \N$,
\[\left \Vert T^n_f \frac{\varphi _k}{\Vert \varphi _k \Vert _p} \right \Vert _p^p  < 2^p \Vert T^n_f \varphi _k \Vert _p^p = 2^p\int_X \vert \varphi _k \circ f^n \vert ^p d\mu \leq 2^{2p}  \int_X \vert \varphi \circ f^n \vert ^p d\mu=2^{2p} \Vert T^n_f \varphi \Vert_p^p.  \]
Note that, for each $k > k_{0}$,  $\frac{\varphi _k}{\Vert \varphi _k \Vert _p} \in  S_{L^p(X)}.$ By the first part of the proof and defining $S_k(\varphi)=\frac{\varphi _k}{\Vert \varphi _k \Vert _p},$  it follows that
\[\lim_{n \rightarrow \infty} \left \Vert T^n_f \left [S_k(\varphi) \right ] \right \Vert _p =  \infty, \text{ uniformly on } \varphi \in S_{L^p(X)} \text{ and } k>k_0,\]
 and hence, from the above computations,
\[ \lim_{n\rightarrow \infty} \Vert T_f^n \varphi \Vert_p^p=   \infty, \text{ uniformly on } S_{L^p(X)},\]
meaning that $T_f$ is uniformly positively expansive. \\

(4). Assume $T_f$ uniformly expansive. By assertion $d)$ of \cite[Proposition 19]{BCDMP}, $S_{L^p(X)}={\mathcal A} \cup {\mathcal C}$, where 
\[\lim_{n \rightarrow \infty} \Vert T_f^n \varphi \Vert_p =\infty \text{ uniformly on ${\mathcal A}$}\] and 
\[\lim_{n \rightarrow \infty} \Vert T_f^{-n}\varphi \Vert_p =\infty \text{ uniformly on ${\mathcal C}$.}\]
Clearly, this implies ${\mathcal B}^+={\mathcal B}^+_{\mathcal A} \cup {\mathcal B}^+_{\mathcal C}$, where 
\[ {\mathcal B}^+_{\mathcal A}= \left \{ B \in {\mathcal B}^+\,:\,\frac{\chi _B}{\mu (B)^{\frac{1}{p}}} \in {\mathcal A}\right  \} \ \  \text{ and } \ \ {\mathcal B}^+_{\mathcal C}= \left \{  B \in {\mathcal B}^+ \,:\,\frac{\chi _B}{\mu (B)^{\frac{1}{p}}} \in {\mathcal C}  \right \}\] are such that  
\[ \lim_{n \rightarrow  \infty} \dfrac{\mu (f^{n}(B))}{\mu (B)}= \infty \text{  uniformly on ${\mathcal B}^+_{\mathcal A}$}\] and  
\[ \lim_{n \rightarrow  \infty} \dfrac{\mu (f^{-n}(B))}{\mu (B)}= \infty \text{ uniformly on ${\mathcal B}^+_{\mathcal C}$},\]
i.e., the thesis holds. \\
To show the other direction in (4), it is sufficient to prove, using again $d)$ of \cite[Proposition 19]{BCDMP}, the existence of $\mathcal A$ and $\mathcal C$ such that $S_{L^p(X)}={\mathcal A} \cup {\mathcal C}$, with
\[\lim_{n \rightarrow \infty} \Vert T_f^n \varphi \Vert_p =\infty \text{ uniformly on ${\mathcal A}$}\] and 
\[\lim_{n \rightarrow \infty} \Vert T_f^{-n}\varphi \Vert_p =\infty \text{ uniformly on ${\mathcal C}$.}\]
Let $M>0$. By hypothesis, there exists $m \in \N$ such that, for all functions of type $\varphi=\dfrac{\chi _B}{\mu (B)^{\frac{1}{p}}}$, with $B \in {\mathcal B}^+$,
\[ \Vert T^n_f \varphi \Vert _p^p \geq M \hspace{0,3 cm} \text{or} \hspace{0,3 cm} \Vert T^{-n}_f \varphi \Vert _p^p \geq M  \hspace{0,3 cm}\forall n \geq m. \] 
Next, the above conclusion is proved for any simple function $\varphi \in S_{L^p(X)}$, and then an approximation by simple functions will provide that it holds also for an arbitrary $\varphi \in S_{L^p(X)}$.  Let $\tilde{S}_{L^p(X)}$ be the collection of simple functions in $S_{L^p(X)}$. First, we hence find two sets of simple functions in ${\tilde{S}}_{L^p(X)}$, denoted $\tilde{\mathcal A}$ and  $\tilde{\mathcal C}$, such that one has ${\tilde{S}}_{L^p(X)}=\tilde{\mathcal A} \cup \tilde{\mathcal C}$, with
\[\lim_{n \rightarrow \infty} \Vert T_f^n \varphi \Vert_p =\infty \text{ uniformly on $\tilde{\mathcal A}$,}\] and 
\[\lim_{n \rightarrow \infty} \Vert T_f^{-n}\varphi \Vert_p =\infty \text{ uniformly on $\tilde{\mathcal C}$.}\]
By hypothesis, in correspondence of $ M>0$, there exists ${\overline n} \in \Bbb N $ such that, for each $n \geq \overline n$,
\[ \dfrac{\mu (f^{n}(B))}{\mu (B)} >  M, \text{ for each $B \in {\mathcal B}^+_{\mathcal A}$,}\]
and 
\[ \dfrac{\mu (f^{-n}(B))}{\mu (B)} >  M, \text{ for each $B \in {\mathcal B}^+_{\mathcal C}$.} \]
Let $\varphi \in {\tilde{S}}_{L^p(X)}$, i.e., $\varphi = \sum_{i=1}^{s} {\alpha}_{i} {\chi}_{B_{i}}$, where $B_i$ are pairwise disjoint measurable sets and $\Vert \varphi \Vert _p^p = \sum_{i=1}^{s} \vert{\alpha _i}\vert^p {\mu (B_{i})} =1$. Without loss of generality, let $B_i \in {\mathcal B}^+$, for each $i \in \{1,...,s\}$. Write $\varphi  =  {\varphi}_{{\mathcal B}^+_{\mathcal A} }+   {\varphi}_{{\mathcal B}^+_{\mathcal C}}$, where 
\[{\varphi}_{{\mathcal B}^+_{\mathcal A}} = \sum_{i \in \{1, \ldots, s\}: B_{i} \in {\mathcal B}^+_{\mathcal A} } {\alpha}_{i} {\chi}_{B_{i}}\]
and 
\[{\varphi}_{{\mathcal B}^+_{\mathcal C}} = \sum_{i \in \{1, \ldots, s\}: B_{i} \in {{\mathcal B}^+_{\mathcal C}} } {\alpha}_{i} {\chi}_{B_{i}}\]
As 
\[\Vert \varphi \Vert _p^p = \sum_{i=1}^{s} \vert{\alpha _i}\vert^p {\mu (B_{i})} = \Vert {\varphi}_{{\mathcal B}^+_{\mathcal A}} \Vert _p^p + \Vert {\varphi}_{{\mathcal B}^+_{\mathcal C}} \Vert _p^p = 1,\]
at least one of these two things must happen 
\[ (a) \ \Vert {\varphi}_{{\mathcal B}^+_{\mathcal A}} \Vert _p^p \geq \frac{1}{2}; \ \ (b) \ \Vert {\varphi}_{{\mathcal B}^+_{\mathcal C}} \Vert _p^p \geq \frac{1}{2}. \]
In case $(a)$,  for each $n \geq \overline n$, 
\begin{eqnarray*}
 \Vert T_f^{-n} \varphi \Vert_p^p  =  \sum_{i=1}^{s} \vert{{\alpha}_{i}}\vert ^p   \mu (f^{n}(B_i))  &  \geq &  \sum_{i \in \{1, \ldots, s\}: B_{i} \in {{\mathcal B}^+_{\mathcal A}} } \vert{{\alpha}_{i}}\vert ^p  \mu (f^{n}(B_i)) \\
& > & M  \sum_{i \in \{1, \ldots, s\}: B_{i} \in {{\mathcal B}^+_{\mathcal A}} } \vert{{\alpha}_{i}}\vert ^p \mu (B_i) \\
& \geq &   M\dfrac{1}{2} . \\
\end{eqnarray*}
In case $(b)$, for each $n \geq \overline n$, 
\begin{eqnarray*}
 \Vert T_f^{n} \varphi \Vert_p^p  =  \sum_{i=1}^{s} \vert{{\alpha}_{i}}\vert ^p   \mu (f^{-n}(B_i)) &  \geq &  \sum_{i \in \{1, \ldots, s\}: B_{i} \in {{\mathcal B}^+_{\mathcal C}} } \vert{{\alpha}_{i}}\vert ^p  \mu (f^{-n}(B_i)) \\
& > & M  \sum_{i \in \{1, \ldots, s\}: B_{i} \in {{\mathcal B}^+_{\mathcal C}} } \vert{{\alpha}_{i}}\vert ^p \mu (B_i) \\
& \geq &   M\dfrac{1}{2} . \\
\end{eqnarray*}
From the above, it follows $\tilde{S}_{L^p(X)} = \tilde{{\mathcal A}} \cup \tilde{{\mathcal C}}$, where 
\[\tilde{{\mathcal A}} = \left \{\varphi \in \tilde{S}_{L^p(X)}: \Vert {\varphi}_{\mathcal{B^+_A}} \Vert _p^p \geq \frac{1}{2} \right \} \ \  \text{ and }  \ \ \tilde{{\mathcal{C}}} = \left \{\varphi \in \tilde{S}_{L^p(X)}: \Vert {\varphi}_{\mathcal{B^+_C}} \Vert _p^p \geq \frac{1}{2} \right \}.\]
Hence, the thesis is proved for $\tilde{S}_{L^p(X)}$, i.e., for simple maps in $S_{L^p(X)}$.   \\
Now, let $\varphi$ be an arbitrary element of $S_{L^p(X)}$. Proceeding as in part (3), write $\varphi = {\varphi}^{+} - {\varphi}^{-}$, where ${\varphi}^{+}$  and ${\varphi}^{-}$ are the positive and the negative part of $\varphi$, respectively. Let $\{{\varphi}^{+}_{k}\}_{k \in \N}$ and $\{{\varphi}^{-}_{k}\}_{k \in \N}$ be two non-decreasing sequences of simple functions, converging pointwise to ${\varphi}^{+}$ and ${\varphi}^{-}$, respectively. Then, the sequence  $\{{\varphi}_{k}\}_{k \in \N}$ defined as ${\varphi}_{k} = {\varphi}^{+}_{k} - {\varphi}^{-}_{k}$, pointwise converges to $\varphi$, and $\vert {\varphi}_{k} \vert  = \vert {\varphi}^{+}_{k} - {\varphi}^{-}_{k} \vert \leq 2 \vert \varphi \vert.$
By the Lebesgue Dominated Convergence Theorem, 
\[\lim_{k \rightarrow  \infty} \Vert \varphi _k \Vert _p = \Vert \varphi  \Vert _p =1,\]
and, hence, there exists $k_{0}$ such that, for each $k > k_{0}$, $\Vert \varphi _k \Vert _p > \frac{1}{2}$. 
Then, for each $k > k_{0}$, $n \in \N$,
\[\left \Vert T^n_f \frac{\varphi _k}{\Vert \varphi _k \Vert _p} \right \Vert _p^p  < 2^p \Vert T^n_f \varphi _k \Vert _p^p = 2^p\int_X \vert \varphi _k \circ f^n \vert ^p d\mu \leq 2^{2p}  \int_X \vert \varphi \circ f^n \vert ^p d\mu=2^{2p} \Vert T^n_f \varphi \Vert_p^p.\]
At least one of these two sets of indexes must be infinite
\[ I_1(\varphi)=\left \{k \in {\Bbb N}: \frac{\varphi _k}{\Vert \varphi _k \Vert_p} \in \tilde{{\mathcal A}} \right \}; \ \ \ \ \ \ \ I_2(\varphi) = \left \{k \in {\Bbb N}: \frac{\varphi _k}{\Vert \varphi _k \Vert_p} \in \tilde{{\mathcal C}}\right \}.\]
In case $I_1(\varphi)$ is infinite, there is an increasing sequence of integers $\{k_j\}$ such that one has $\left \{ \frac{\varphi_{k_{j}}}{\Vert \varphi_{k_{j}} \Vert_p}\right \} \subseteq \tilde{{\mathcal A}}$ and, by the first part of the proof,  for each $M>0$, there exists $n_0 \in \Bbb N$ such that
\[ \left \Vert T^{-n}_f \frac{\varphi_{k_{j}}}{\Vert \varphi_{k_{j}} \Vert_p} \right  \Vert _p^p > \frac{M}{2}, \text{ for each $n > n_0,$} \]
and then 
\[\Vert T^{-n}_f \varphi \Vert _p^p > \frac{1}{2^{2p}} \left  \Vert T^{-n}_f \frac{\varphi_{k_{j}}}{\Vert \varphi_{k_{j}} \Vert_p}  \right \Vert _p^p > \frac{M}{2^{2p+1}}, \text{ for each $n > n_0$}.\]
In case $I_2(\varphi)$ is infinite, there is an increasing sequence of integers $\{k_j\}$ such that one has $\left \{ \frac{\varphi_{k_{j}}}{\Vert \varphi_{k_{j}} \Vert_p}\right \} \subseteq \tilde{{\mathcal C}}$ and, by the first part of the proof,  for each $M>0$,  there exists $n_0 \in \Bbb N$ such that
\[\left \Vert T^{n}_f \frac{\varphi_{k_{j}}}{\Vert \varphi_{k_{j}} \Vert_p} \right \Vert _p^p > \frac{M}{2},  \text{ for each $n > n_0,$}\]
and then 
\[\Vert T^{n}_f \varphi \Vert _p^p > \frac{1}{2^{2p}} \left \Vert T^{n}_f \frac{\varphi_{k_{j}}}{\Vert \varphi_{k_{j}} \Vert_p} \right \Vert _p^p > \frac{M}{2^{2p+1}},  \text{ for each $n > n_0$}.\]
Letting \[{\mathcal A}= \left \{ \varphi \in S_{L^p(X)} : \# I_2(\varphi)=\infty \right \} \text{ and }\  {\mathcal C}=\left \{ \varphi \in S_{L^p(X)} :  \# I_1(\varphi)=\infty  \right \},\] 
it follows that $S_{L^p(X)}={\mathcal A} \cup \mathcal C$, with
\[\lim_{n \rightarrow \infty} \Vert T_f^n \varphi \Vert_p =\infty \text{ uniformly on ${\mathcal A}$,}\] and 
\[\lim_{n \rightarrow \infty} \Vert T_f^{-n}\varphi \Vert_p =\infty \text{ uniformly on ${\mathcal C}$,}\]
and hence, the thesis. 
\end{proof}

\begin{rmk}\label{rmkB} Note that the statement $(2)$ of the previous theorem remains true if one replaces ``for every $B \in \mathcal B$ with positive measure'' with ``for every $B \in \mathcal B^+$'', i.e., statement $(2)$ can be rewritten as `` $T_f$ is expansive if and only if for every $B \in \mathcal B^+$, $\sup_{n \in \Bbb Z}\mu (f^{-n}(B))= \infty$''.
\end{rmk}

The following conditions hold in the specific setting of dissipativity with bounded distortion, and they will be used in the sequel.

\begin{manualtheorem}{ED} \label{thmEdissipative}
Let $(X,{\mathcal B},\mu, f, T_f)$ be a dissipative composition dynamical system of bounded distortion, generated by $W$. Then, the following hold.
\begin{itemize}
\item[(1)]{$T_f$ is positively expansive if and only if  $\displaystyle{\sup_{n \in \mathbb N}\mu (f^{-n}(W))=\infty}$.}
\item[(2)]{$T_f$ is expansive if and only if  $\displaystyle{\sup_{n \in \mathbb Z}\mu (f^{n}(W))=\infty}$.}
\item[(3)]{$T_f$ is uniformly positively expansive if and only if $\displaystyle{\lim_{n \rightarrow \infty } \mu(f^{-n}(W))= \infty.}$}
\item[(4)]{$T_f$ is uniformly expansive if and only if one of the following conditions holds:
\begin{equation}
\lim_{n \rightarrow \infty} \inf _{k \in {\mathbb Z}} \left (\frac{\mu(f^{k+n}(W))}{\mu(f^{k}(W))}\right )  = \infty \tag*{$\mathcal{UE}1$}   
\end{equation}
\begin{equation}
 \lim_{n \rightarrow \infty} \inf_{k \in {\mathbb Z}} \left (\frac{\mu(f^{k-n}(W))}{\mu(f^{k}(W))} \right ) = \infty \tag*{$\mathcal{UE}2$} 
\end{equation}
\begin{equation}
 \lim_{n \rightarrow \infty} \inf _{k \in {\mathbb N}} \left (\frac{\mu(f^{k+n}(W))}{\mu(f^{k}(W))}\right  )  = \infty
 \   \ \& \ \ 
\lim_{n \rightarrow \infty} \inf_{k \in -{\mathbb N}_0} \left (\frac{\mu(f^{k-n}(W))}{\mu(f^{k}(W))}\right ) = \infty \tag*{$\mathcal{UE}3$}
\end{equation}
}
\end{itemize}
\end{manualtheorem}

\begin{proof}
The proof of (1) is skipped as it follows from the proof of (2), by replacing the set $\Z$ with $\N$. \\

$(2)$. By $(2)$ of Theorem \ref{PROPEX1} and Remark \ref{rmkB}, the following equivalence has to be proved
\[``\text{$\forall B \in {\mathcal B^+}, $}\sup_{n \in \mathbb Z}\mu (f^{-n}(B))= \infty  \Leftrightarrow \sup_{n \in \mathbb Z}\mu (f^{-n}(W))=\infty".\]
The implication $``\Rightarrow"$ is obvious. To show the other one, let $B \in {\mathcal B}^+$. Hence, there exists $n_{0} \in {\mathbb Z}$ such that $\mu(B \cap f^{n_{0}}(W)) >0$. Take  $A := f^{- n_{0}}(B \cap f^{n_{0}}(W)) = f^{-n_{0}}(B) \cap W$. Then, $A \in {\mathcal B}(W)$, $\mu(A)>0$ and, by applying the bounded distortion property, given by Condition ($\Diamond$) on page 6, it follows that
\[\sup_{n \in \mathbb Z}\mu (f^{-n}(A)) \geq \dfrac{1}{K} \dfrac{\mu(A)}{\mu(W)} \sup_{n \in \mathbb Z} \mu (f^{-n}(W))= \infty\]
and, hence,
\[ \sup_{n \in \mathbb Z}\mu (f^{-n}(B)) \geq \sup_{n \in \mathbb Z}\mu (f^{-n}(A))= \infty.\] 
Therefore, the thesis holds.\\

$(3).$ By $(3)$ of Theorem \ref{PROPEX1}, it has to be showed that 
\[\lim_{n \rightarrow  \infty} \dfrac{ \mu (f^{-n}(B))}{\mu(B)}= \infty \text{ uniform. w.r.t. $B \in {\mathcal B}^+$} \Leftrightarrow \lim_{n \rightarrow \infty } \mu(f^{-n}(W))= \infty\]
where we recall that ${\mathcal B}^+=\{B \in {\mathcal B} : 0 < \mu(B) < \infty \}.$
The implication $``\Rightarrow"$ is obvious. To see the converse, assume $\displaystyle{\lim_{n \rightarrow \infty } \mu(f^{-n}(W))= \infty}$ and fix $B \in {\mathcal B}^+$. As in the proof of (2), there exists $n_{0} \in {\mathbb Z}$ such that $B \cap f^{n_{0}}(W) \in {\mathcal B}^+$. Let $A = f^{- n_{0}}(B \cap f^{n_{0}}(W)) = f^{-n_{0}}(B) \cap W$. Then, $A \in {\mathcal B}(W)$, $0<\mu(A)<\infty$, and by using the bounded distortion property, given by Condition {($\Diamond$)} on page 6, one obtains
\[\dfrac{\mu (f^{-n}(A))}{\mu (A)} \geq \dfrac{1}{K} \dfrac{ \mu (f^{-n}(W))}{\mu (W)}, \ \ \text{ for each } n \in \mathbb Z,\]
where $K$ is the bounded distortion constant.
Hence, for every $n \in \mathbb Z$, 
 \[\mu(f^{-n-n_0}(B)) \geq \mu(f^{-n}(A))\geq \dfrac{1}{K} \dfrac{\mu (f^{-n}(W))}{\mu (W)} \mu(A)\]
 implying  
 \[\lim_{n \rightarrow \infty}\mu(f^{-n-n_0}(B)) \geq  \dfrac{1}{K \mu (W)}\lim_{n \rightarrow \infty} \mu (f^{-n}(W))=\infty .\]
By the arbitrariness of $B \in \mathcal B^+$ and writing $m=n+n_0$, it follows that
\[\lim_{m \rightarrow \infty} \dfrac{\mu (f^{-m}(B))}{\mu (B)}=\infty, \ \ \text{ uniformly w.r.t. } B \in {\mathcal B}^+, \]
showing, by (3) of Theorem \ref{PROPEX1}, that $T_f$ is uniformly positively expansive.\\

$(4).$ Assume $T_f$ uniformly expansive. By $(4)$ of Theorem \ref{PROPEX1}, ${\mathcal B}^+$ can be splitted as ${\mathcal B}^+={\mathcal B}^+_{\mathcal A} \cup {\mathcal B}^+_{\mathcal C}$, where 
\[\lim_{n \rightarrow  \infty} \dfrac{\mu (f^{n}(B))}{\mu (B)}= \infty \hspace{0,3 cm} \text{ uniformly on ${\mathcal B}^+_\mathcal A$} \] and  
\[\lim_{n \rightarrow  \infty} \dfrac{\mu (f^{-n}(B))}{\mu (B)}= \infty \hspace{0,3 cm} \text{ uniformly on ${\mathcal B}^+_\mathcal C$.}\] 
Let 
\[I = \left \{k \in \mathbb Z: f^{k}(W) \in {\mathcal B}^+_{\mathcal A}\right\} \text{ and } J = \left \{k \in \mathbb Z: f^{k}(W) \in {\mathcal B}^+_{\mathcal C} \right\}.\] 
Then,  
\[\lim_{n \rightarrow  \infty} \inf _{k \in I} \left (\frac{\mu(f^{k+n}(W))}{\mu(f^{k}(W))}\right )  = \infty\] 
and 
\[\lim_{n \rightarrow  \infty} \inf_{k \in { J}} \left (\frac{\mu(f^{k-n}(W))}{\mu(f^{k}(W))}\right ) = \infty.\]
Now, if $J=\emptyset$, then Condition ${\mathcal{UE}}1$ holds and, if $I=\emptyset$, then Condition ${\mathcal{UE}}2$ holds.
On the other hand, if $I$ and $J$ are both non-empty, then there exist $i$ and $j$ in 
$\mathbb Z$ such that 
\[[i, + \infty [  \cap {\mathbb Z} \subseteq I \text{ and } ]- \infty, j ] \cap {\mathbb Z} \subseteq J,\]
so that Condition ${\mathcal{UE}}3$ is satisfied. \\
Next, the reverse implication is proved, i.e., it is showed that each of Conditions~$\mathcal{UE}1$, $\mathcal{UE}2$ or $\mathcal{UE}3$ implies the uniform expansivity of $T_f$. To this aim, fix
 $B \in {\mathcal B}^+$. As the system is dissipative, one can write \[B = \dot{\bigcup_{k \in \mathbb Z}} B_{k} \ \ \text{ where } \ \ B_{k} = B \cap f^{k}(W). \]
 As the system is of bounded distortion, Condition $(\Diamond\Diamond)$ on page 6 holds, i.e., there exists a constant $H>0$ such that, for each $k$ and  $n \in \Z$,
\[\dfrac{1}{H} \frac{\mu(f^{k+n}(W))}{\mu(f^{k}(W))} {\mu(B_k)}\leq  {\mu(f^{n}(B_k))} \leq H \frac{\mu(f^{k+n}(W))}{\mu(f^{k}(W))}{\mu(B_k)}.\]
As $\mu(B)>0,$ then there exists $B_k$ with $\mu(B_k)>0$. Hence, for each $n \in \Z$,
\begin{align*} 
\mu(f^n(B)) &= \sum_{k \in \Z} \mu(f^n(B_k)) \\
&\geq \dfrac{1}{H} \sum_{k \in \Z} \frac{\mu(f^{k+n}(W))}{\mu(f^{k}(W))}  {\mu(B_k)} \\
&\geq  \dfrac{1}{H} {\mu(B)} \inf_{k \in {\Z} } \frac{\mu(f^{k+n}(W))}{\mu(f^{k}(W))}. \\
\end{align*}
This implies that, if one of Conditions~$\mathcal{UE}1$, $\mathcal{UE}2$ or $\mathcal{UE}3$ holds, then  
\[\lim_{n \rightarrow  \infty} \dfrac{\mu (f^{n}(B))}{\mu (B)}= \infty \hspace{0,3 cm} \text{ uniformly on ${\mathcal B}^+_{\mathcal A}$} \]
and 
\[\lim_{n \rightarrow  \infty} \dfrac{\mu (f^{-n}(B))}{\mu (B)}= \infty \hspace{0,3 cm} \text{ uniformly on ${\mathcal B}^+_{\mathcal C}$} \]
where ${\mathcal B}^+={\mathcal B}^+_{\mathcal A} \cup {\mathcal B}^+_{\mathcal C}$ with 
\[ {\mathcal B}^+_{\mathcal A} = {\mathcal B}^+ \cap (  \cup_{k \geq 0} f^{k}(W) )=  \{ B \cap (  \cup_{k \geq 0} f^{k}(W) ) :  B \in \mathcal B^+  \}\] and \[ {\mathcal B}^+_{\mathcal C}= {\mathcal B}^+ \cap \left (\cup_{k < 0} f^{k}(W)\right)=\left \{ B \cap (  \cup_{k < 0} f^{k}(W) ) :  B \in \mathcal B^+ \right \}.\]
\end{proof}

\section{Strong Structural Stability for Composition Operators}

\begin{prop} \label{lem2}
Let $(X,{\mathcal B},\mu, f, T_f)$ be a dissipative composition dynamical system of bounded distortion, generated by $W$.  If the following condition holds
\[ \uplim_{n \rightarrow \infty} \sup _{k \in \Bbb N_0}\left( \dfrac{\mu (f^{k}(W))}{\mu (f^{k+n}(W))} \right)^{\frac{1}{n}}<1 \ \ \& \ \ \lowlim_{n \rightarrow \infty} \inf _{k \in -\Bbb N_0}\left( \dfrac{\mu (f^{k-n}(W))}{\mu (f^{k}(W))} \right)^{\frac{1}{n}}>1,\]
then the operator $T_f$ is not structurally stable and, hence, not even strongly structurally stable.
\end{prop}
\begin{proof}
By the second inequality in the hypothesis, there exists $n_0 \in \N$ and $l>1$ such that 
\[ \inf _{k \in -\Bbb N_0}\left( \dfrac{\mu (f^{k-n}(W))}{\mu (f^{k}(W))} \right)^{\frac{1}{n}}\geq l >1, \text{ for each } n \geq n_0. \tag{$\heartsuit$}\]
This implies, taking $k=0$,
\[ \dfrac{\mu(f^{-n}(W))}{\mu(W)} \geq \inf _{k \in -\Bbb N_0}\dfrac{\mu (f^{k}(W))}{\mu (f^{k-n}(W))} \geq l^n, \text{ for each } n \geq n_0 \]
and hence $\sup_{n \in \N} \mu(f^{-n}(W)) \geq \sup_{n \in [n_0, \infty)} \mu(f^{-n}(W))= \infty$, i.e., by point (1) of Theorem \ref{thmEdissipative}, $T_f$ is positively expansive.
Moreover, from $(\heartsuit)$ it follows that
\[ \sup_{k \in - \Bbb N_0} \left (\dfrac{\mu (f^{k-n}(W))}{\mu (f^{k}(W))}\right )^{\frac{1}{n}} \geq \inf _{k \in -\Bbb N_0}\left( \dfrac{\mu (f^{k-n}(W))}{\mu (f^{k}(W))} \right)^{\frac{1}{n}} \geq l >1, \text{ for each } n \geq n_0,\]
implying 
\[ \uplim_{n \rightarrow \infty} \sup_{k \in  \Bbb Z} \left (\dfrac{\mu (f^{k-n}(W))}{\mu (f^{k}(W))}\right )^{\frac{1}{n}} \geq \uplim_{n \rightarrow \infty} \sup_{k \in - \Bbb N_0} \left (\dfrac{\mu (f^{k-n}(W))}{\mu (f^{k}(W))}\right )^{\frac{1}{n}} >1. \tag{$\heartsuit \heartsuit$}\]

Now, the first inequality in the hypothesis gives that there exists $m_0 \in \N$ and $0<t<1$ such that
\[ \sup_{k \in \Bbb N_0}\left( \dfrac{\mu (f^{k}(W))}{\mu (f^{k+n}(W))} \right)^{\frac{1}{n}} \leq t <1, \text{ for each } n \geq m_0. \tag{$\circ$}\]
This implies that, for each $n \geq m_0$,
\begin{eqnarray*}
 \inf_{k \in \Z} \left ( \dfrac{\mu (f^{k}(W))}{\mu (f^{k+n}(W))} \right)^{\frac{1}{n}}  \leq \inf_{k \in \N _0}  \left ( \dfrac{\mu (f^{k}(W))}{\mu (f^{k+n}(W))} \right)^{\frac{1}{n}}  \leq \sup_{k \in \N_0}  \left (\dfrac{\mu (f^{k}(W))}{\mu (f^{k+n}(W))} \right)^{\frac{1}{n}} \leq t <1
\end{eqnarray*}
implying, 
\[ \lowlim_{n \rightarrow \infty}\inf_{k \in \Z} \left ( \dfrac{\mu (f^{k}(W))}{\mu (f^{k+n}(W))}  \right)^{\frac{1}{n}} \leq  \lowlim_{n \rightarrow \infty} \inf_{k \in \N _0} \left (  \dfrac{\mu (f^{k}(W))}{\mu (f^{k+n}(W))}  \right)^{\frac{1}{n}}  <1. \tag{$\circ \circ$}\]
From $(\circ \circ)$ and $(\heartsuit \heartsuit)$ it follows that none of Conditions $\mathcal{HC, HD, GH}$ of \cite[Corollary SC]{DAnielloDarjiMaiuriello} holds, precisely: $(\circ \circ)$  implies that the second half of Condition $\mathcal{GH}$ (and hence Condition $\mathcal{HD}$) does not hold, while $(\heartsuit \heartsuit)$ implies that the first half of Condition $\mathcal{GH}$ and Condition $\mathcal{HC}$ do not hold. This means that $T_f$ does not have the shadowing property and, therefore, it is not hyperbolic. As $T_f$ is positively expansive but not hyperbolic, then, using $(b)$ of Theorem \ref{SS}, it follows that $T_f$ is not structurally stable and, therefore, not even strongly structurally stable.
\end{proof}

\begin{manualtheorem}{SC1} \label{SC1}
Let $(X,{\mathcal B},\mu, f, T_f)$ be a dissipative composition dynamical system of bounded distortion, generated by $W$. If one of the following conditions holds: 
\begin{equation*}\label{hc}
\uplim_{n \rightarrow \infty} \sup _{k \in {\mathbb Z}} {\left (\frac{\mu(f^{k}(W))}{\mu(f^{k+n}(W))}\right )}^{\frac{1}{n}} <1 \tag*{$\mathcal{HC}$}   
\end{equation*}
\begin{equation*}\label{hd}
 \lowlim_{n \rightarrow \infty} \inf _{k \in {\mathbb Z}} {\left (\frac{\mu(f^{k}(W))}{\mu(f^{k+n}(W))} \right)}^{\frac{1}{n}} >1 \tag*{$\mathcal{HD}$} 
\end{equation*}
\begin{equation*}{\label{gh}
  \uplim_{n \rightarrow  \infty} \sup _{k \in -{\mathbb N}_{0}} { \left (\frac{\mu(f^{k-n}(W))}{\mu(f^{k}(W))} \right )}^{\frac{1}{n}} < 1 
  \ \  \& \ \ 
\lowlim_{n \rightarrow  \infty} \inf_{k \in {\mathbb N_{0}} } {\left (\frac{\mu(f^{k}(W))}{\mu(f^{k+n}(W))}\right )}^{\frac{1}{n}} >1 \tag*{$\mathcal{GH}$}}
\end{equation*}
then the operator $T_f$ is strongly structurally stable.
\end{manualtheorem}
\begin{proof}
Assume that one of Conditions $\mathcal{HC}$, $\mathcal{HD}$ or $\mathcal{GH}$ holds. Note that, using \cite[Theorem SS]{DAnielloDarjiMaiuriello}, if Condition $\mathcal{HC}$, $\mathcal{HD}$ holds, then $T_f$ is hyperbolic and, hence, generalized hyperbolic; if Condition $\mathcal{GH}$ holds, then $T_f$ is generalized hyperbolic. The thesis follows by applying Theorem \ref{genhyp}.\\
\end{proof}

\begin{manualcor}{SC2} \label{SC2}
Let $(X,{\mathcal B},\mu, f, T_f)$ be a dissipative composition dynamical system of bounded distortion, generated by $W$.  If the operator $T_f$ has the shadowing property or, equivalently, is generalized hyperbolic, then it is strongly structurally stable.
\end{manualcor}
\begin{proof}
It simply follows comparing Theorem \ref{SC1} with \cite[Corollary SC]{DAnielloDarjiMaiuriello}.\\
\end{proof}

\begin{manualtheorem}{W} \label{W}
Let $(X,{\mathcal B},\mu, f, T_f)$ be a dissipative composition dynamical system of bounded distortion, generated by $W$. Consider the weighted backward shift $B_w$ on $\ell^p(\Z)$ with weights
\[w_{k} =  \left( \frac{\mu(f^{k-1}(W))}{\mu(f^{k}(W))}\right)^{\frac{1}{p}} .\] If $B_w$ is strongly structurally stable, then so is $T_f$.
\end{manualtheorem}
\begin{proof}
If $B_w$ is strongly structurally stable, then one of conditions $a), b), c)$ of Theorem \ref{theoSSBW} holds. In particular, $B_w$ has the shadowing property. Using the fact that  $w_{k} =  \left( \frac{\mu(f^{k-1}(W))}{\mu(f^{k}(W))}\right)^{\frac{1}{p}}, $
and using \cite[Theorem M]{DAnielloDarjiMaiuriello2}, it follows that $T_f$ has the shadowing property and hence one of conditions $\mathcal{HC}$, $\mathcal{HD}$ or $\mathcal{GH}$ of Theorem \ref{SC1} holds. This implies the thesis.\\
\end{proof}

\begin{manualtheorem}{C} \label{C}
Let $(X,{\mathcal B},\mu, f, T_f)$ be a dissipative composition dynamical system of bounded distortion, generated by $W$. Assume that $\sup_{n \in \N}\mu(f^{-n}(W))= \infty$. Then $T_f$ is strongly structurally stable if and only if it has the shadowing property. 
\end{manualtheorem}
\begin{proof}
$(\Leftarrow)$ If $T_f$ has the shadowing property, then by Corollary \ref{SC2} it follows that $T_f$ is strongly structurally stable.\\

$(\Rightarrow)$ As $T_f$ is strongly structurally stable, then it is structurally stable. By hypothesis, $\sup_{n \in \N}\mu(f^{-n}(W))= \infty$, meaning, by Theorem \ref{thmEdissipative}, that $T_f$ is positively expansive. Hence, Theorem \ref{SS} implies $T_f$ hyperbolic and then the thesis.
\end{proof}

{\bf{Open Problem:}} Does the equivalence hold in Theorem \ref{SC1} (and hence in Corollary SC2 and Theorem W)?\\

\bibliography{bibliothesis}
\bibliographystyle{siam}

\Addresses

\end{document}